\documentclass[amscd,amssymb,11pt]{amsart}

\usepackage{amssymb}
\usepackage[all]{xy}

\newtheorem{thm}{Theorem}[section]
\newtheorem{lem}[thm]{Lemma}
\newtheorem{cor}[thm]{Corollary}
\newtheorem{prop}[thm]{Proposition}

\newtheorem{prop/def}[thm]{Proposition/Definition}

\theoremstyle{definition}

\newtheorem{defn}[thm]{Definition}

\newtheorem{ex}[thm]{Example}

\theoremstyle{remark}

\newtheorem{rem}[thm]{Remark}

\numberwithin{equation}{section}

\newcommand{\thmref}[1]{Theorem~\ref{#1}}
\newcommand{\corref}[1]{Corollary~\ref{#1}}
\newcommand{\secref}[1]{\S\ref{#1}}

\newcommand{\propref}[1]{Proposition~\ref{#1}}
\newcommand{\lemref}[1]{Lemma~\ref{#1}}

\newcommand{\exref}[1]{Example~\ref{#1}}
\newcommand{\remref}[1]{Remark~\ref{#1}}

\newcommand{\Hom}{\operatorname{Hom}}

\newcommand{\Ind}{\operatorname{Ind}}
\newcommand{\Ext}{\operatorname{Ext}}
\newcommand{\coker}{\operatorname{coker}}

\newcommand{\Id}{\operatorname{Id}}

\newcommand{\A}{{\mathcal  A}}
\newcommand{\B}{{\mathcal  B}}
\newcommand{\C}{{\mathcal  C}}

\newcommand{\K}{{\mathcal  K}}
\newcommand{\U}{{\mathcal  U}}
\newcommand{\Nil}{{\mathcal Nil}}
\newcommand{\F}{{\mathcal  F}}

\newcommand{\bT}{{\bar  T}}

\newcommand{\Z}{{\mathbb  Z}}

\newcommand{\ra}{\rightarrow}
\newcommand{\xra}{\xrightarrow}

\newcommand{\lra}{\longrightarrow}
\newcommand{\hra}{\hookrightarrow}

\begin{document}

\title[On the Krull filtration of $\U$]{The Krull filtration of the category of unstable modules over the Steenrod algebra}

\author[Kuhn]{Nicholas J.~Kuhn}
\address{Department of Mathematics \\ University of Virginia \\ Charlottesville, VA 22904}
\email{njk4x@virginia.edu}
\thanks{This research was partially supported by grants from the National Science Foundation}

\date{June 25, 2013.}


\begin{abstract}  In the early 1990's, Lionel Schwartz gave a lovely characterization of the Krull filtration of $\U$, the category of unstable modules over the mod $p$ Steenrod algebra.  Soon after, this filtration was used by the author as an organizational tool in posing and studying some topological nonrealization conjectures.  In recent years the Krull filtration of $\U$ has been similarly used by Castellana, Crespo, and Scherer in their study of H--spaces with finiteness conditions.  Recently, Gaudens and Schwartz have given a proof of some of my conjectures.  In light of these topological applications, it seems timely to better expose the algebraic properties of the Krull filtration.
\end{abstract}

\maketitle
\section{Introduction} \label{introduction}

The mod $p$ cohomology of a space $H^*(X)$ is naturally an object in both $\U$ and $\K$, the categories of unstable modules and algebras over the mod $p$ Steenrod algebra $\A$. Over the past quarter century, much important work in unstable homotopy theory has been grounded in the tremendous advances made during the 1980's in our understanding of the abelian category $\U$, and the central role played by $H^*(B\Z/p)$.

In his 1962 paper [Gab62], Gabriel introduced the Krull filtration of a general abelian category.  Applied to $\U$, the Krull filtration is the increasing sequence of localizing subcategories
$$ \U_0 \subset \U_1 \subset \U_2 \subset \U_3 \subset \U_4 \subset \dots$$
recursively defined as follows:  $\U_0$ is the full subcategory of locally finite unstable $\A$-modules, and for $n>0$, $\U_n$ is the full subcategory of unstable modules which project to a locally finite object in the quotient category $\U/\U_{n-1}$.

In the early 1990's, Lionel Schwartz gave a lovely characterization of the Krull filtration of $\U$ in terms of Lannes' $T$-functor.

Soon after, this filtration was used by the author \cite{k1} as an organizational tool in posing and studying some topological nonrealization conjectures.  In particular, the author's Strong Realization Conjecture posited that if $H^*(X) \in \U_n$ for some finite $n$, then $H^*(X) \in \U_0$.  A proof of this was recently given by Schwartz and G. Gaudens \cite{gaudens schwartz}.  (Special cases were proved earlier in \cite{k1,s3,s4}.)

The Krull filtration of $\U$ has been similarly used in recent years by Castellana, Crespo, and Scherer in their study of H--spaces with finiteness conditions \cite{ccs1,ccs2}.

The few results about the Krull filtration are scattered in the literature, and are not particularly comprehensive.  Moreover, Schwartz just sketched a proof of his key characterization in \cite[\S6.2]{s2}, and we suspect many readers would have trouble filling in the details.  In light of the recent topological applications, it seems timely to better expose the properties of this interesting bit of the structure of $\U$.

\subsection{Schwartz's characterization of $\U_n$}

Lannes \cite{L} defines $T: \U \ra \U$ to be the left adjoint to the functor sending $M$ to $H^*(B\Z/p) \otimes M$, and proves that this functor satisfies many remarkable properties; in particular, it is exact and commutes with tensor products.

We need the reduced version. Let $\bT: \U \ra \U$ be left adjoint to the functor sending $M$ to $\tilde H^*(B\Z/p) \otimes M$, so that $TM = \bT M \oplus M$.  Then let $\bT^n$ denote its $n$th iterate.

Schwartz's elegant characterization of $\U_n$ \cite[Thm.6.2.4]{s2} goes as follows.  It generalizes the $n=0$ case, proved earlier by Lannes and Schwartz \cite{ls2}.

\begin{thm} \label{schwartz thm}
 \label{T thm} $\U_n = \{ M \in \U \ | \ \bT^{n+1}M = 0 \}$. \end{thm}

We will give a proof of this theorem which is simpler than the proof outlined in \cite{s2}.

All of our subsequent results about $\U_n$ use this theorem.

As an example, since $\bT(M \otimes N) = (\bT M \otimes N) \oplus (M \otimes \bT N) \oplus (\bT M \otimes \bT N)$, the theorem has the following first consequence.

\begin{cor}  If $M \in \U_m$ and $N \in \U_n$, then $M \otimes N \in \U_{m+n}$.
\end{cor}

\subsection{The quotient category $\U_n/\U_{n-1}$ and consequences.}

Our next theorem usefully identifies the quotient category $\U_n/\U_{n-1}$.

We need to introduce some basic unstable modules.  Let $F(1)$ be the free unstable module on a $1$--dimensional class.  Explicitly, $F(1) = \A\cdot x \subset H^*(B\Z/p)$, where $x$ generates $H^1(B\Z/p)$. Thus when $p=2$,
$$ F(1) =
\langle x, x^2, x^4, \dots \rangle \subset \Z/2[x],$$
and when $p$ is odd,
$$ F(1) = \langle x, y, y^p, y^{p^2}, \dots \rangle \subset \Lambda(x) \otimes \Z/p[y].$$

If $M \in \U_n$, then $\bT^nM \in \U_0$.  But $\bT^nM$ has more structure than just this: note that $\bT^n$ is left adjoint to tensoring with $\tilde H^*(B\Z/p)^{\otimes n}$, and thus $\bT^n(M)$ has a natural action of the $n$th symmetric group $\Sigma_n$. Let $\Sigma_n\text{--}\U_0$ denote the category of $M \in \U_0$ equipped with a $\Sigma_n$--action.

\begin{thm} \label{Un/Un-1 thm} The exact functor $\bT^n: \U_n \ra \Sigma_n\text{--}\U_0$ has as right adjoint the functor $$N \mapsto (N \otimes F(1)^{\otimes n})^{\Sigma_n},$$ and together these functors induce an equivalence
$$\U_n/\U_{n-1} \simeq \Sigma_n\text{--}\U_0.$$
\end{thm}

Using this theorem, one quite easily obtains a recursive description of modules in $\U_n$.

\begin{thm} \label{recursive Un thm} $M \in \U_n$ if and only if there exist $K,Q \in \U_{n-1}$, $N \in \Sigma_n\text{--}\U_0$, and an exact sequence
$$ 0 \ra K \ra M \ra (N \otimes F(1)^{\otimes n})^{\Sigma_n} \ra Q \ra 0.$$
Furthermore, $\bar T^n M \simeq N$.
\end{thm}

The case $n=1$, already useful, was previously known \cite[Prop.2.3]{s4}.

One obtains a simple generating set for $\U_n$.

\begin{thm} \label{F(1) thm}  $\U_n$ is the smallest localizing subcategory containing all suspensions of the modules $F(1)^{\otimes m}$ for $0 \leq m \leq n$.
\end{thm}

\begin{rem} Let $F(n)$ be the free unstable module on an $n$--dimensional class.  Though $F(n) \in \U_n$, $\U_n$ is generally strictly larger than the localizing subcategory generated by all suspensions of the modules $F(m)$ for $0\leq m \leq n$.  See \exref{F(n) example}.
\end{rem}

\subsection{Interaction with the nilpotent filtration}

To put \thmref{Un/Un-1 thm} in perspective, we remind readers of how the Krull filtration of $\U$ interacts with the nilpotent filtration \cite[\S2]{k1}.

The nilpotent filtration of $\U$, introduced in \cite{s1}, is the decreasing filtration
$$ \U = \Nil_0 \supset \Nil_1 \supset \Nil_2 \supset \Nil_3 \supset \dots$$
defined by letting $\Nil_s$ be the smallest localizing subcategory containing all $s$--fold suspensions of unstable modules.  We write $\Nil$ for $\Nil_1$.

H.-W. Henn, Lannes, and Schwartz \cite{hls1} identify $\U/\Nil$ as follows.  Let $\F$ be the category of functors from finite dimensional $\Z/p$--vector spaces to $\Z/p$--vector spaces.  There is a difference operator $\Delta: \F \ra \F$ defined by
$$ \Delta F(V) = F(V \oplus \Z/p)/F(V).$$
$F$ is said to be polynomial of degree $n$ if $\Delta^{n+1}F = 0$, and analytic if it is locally polynomial.  Let $\F^n$ and $\F^{an}$ denote the subcategories of such functors.

The authors of \cite{hls1} show that there is an equivalence
$ \U/\Nil \simeq \F^{an}$.  Under this equivalence, it is not hard to show that $\bar T$ corresponds to $\Delta$, thus $\U_n$ projects to $\F^n$.
Schwartz \cite{s1} further observes that, for all $s$, $M \mapsto \Sigma^s M$ induces an equivalence $ \U/\Nil \simeq \Nil_s/\Nil_{s+1}$.

Restricting these equivalences to $\U_n$, one learns that, for all $s$, there are equivalences of abelian categories
$$ (\Nil_s \cap \U_n)/(\Nil_{s+1} \cap \U_n)  \simeq \F^n.$$

It is classic (see \cite{pirashvili}, \cite[\S 5.5]{s1}, or \cite{filt gen rep}) that there is an equivalence
\begin{equation} \label{Fn/Fn-1thm} \F^n / \F^{n-1} \simeq \Z/p[\Sigma_n]\text{--modules}.
\end{equation}

The equivalence $\U_n/\U_{n-1} \simeq \Sigma_n\text{--}\U_0$ of \thmref{Un/Un-1 thm} can thus be seen as the correct lift of this result to $\U$.

\begin{rem}  The strongest form of (\ref{Fn/Fn-1thm}) says that the quotient functor $\F^n \ra \Z/p[\Sigma_n]\text{--modules}$ admits both a left and right adjoint, forming a recollement setting \cite[Example 1.5]{filt gen rep}.  By contrast, as $\bar T^n$ does not commute with products, the exact quotient functor
$\U_n \ra \Sigma_n\text{--} \U_0$ does not admit a left adjoint. See \remref{no adjoint remark} for related comments.
\end{rem}

\subsection{The Krull filtration of a module}

Let $k_n: \U \ra \U_n$ be right adjoint to the inclusion $\U_n \hra \U$.  Explicitly, $k_nM \subset M$ is the largest submodule of $M$ contained in $\U_n$.  (So $k_0M$ is the locally finite part of $M$.)  We obtain a natural increasing filtration of $M$:
$$ k_0M \subset k_1M \subset k_2M \subset k_3M \subset \dots .$$
It is useful to let $\bar k_nM$ denote the composition factor $k_nM/k_{n-1}M$.

A basic calculation from  \cite{hls1} can be interpreted as saying the following.

\begin{prop} \label{pnH prop}  \cite[Lem.7.6.6]{hls1} $k_n \tilde H^*(B\Z/p)$ is the span of products of elements in $F(1)$ of length at most $n$.
\end{prop}

The next theorem lists some basic properties.  In this theorem, $\Phi$ is the Frobenius functor, $nil_sM$ is the maximal submodule of $M$ in $\Nil_s$, $R_sM$ is the reduced module defined by $nil_sM/nil_{s+1}M = \Sigma^s R_sM$, and $\bar R_sM$ is its $\Nil$--closure.  (These will be recalled in more detail in \secref{U section}.)

\begin{thm} \label{kn thm}  The Krull filtration of a module satisfies the following properties, for all $M,N \in \U$.

\noindent{\bf (a)} \ $k_n$ is left exact, commutes with filtered colimits, and $\displaystyle \bigcup_{n=0}^{\infty} k_nM = M$. \\

\noindent{\bf (b)} \ $k_n$ commutes with the functors $\Sigma^s$, $\Phi$, $nil_s$, and $\bar R_s$. \\

\noindent{\bf (c)} \ $k_n$ preserves $\Nil$--reduced modules, $\Nil$--closed modules, and \\ $\Nil$--isomorphisms.   \\

\noindent{\bf (d)} \ $\displaystyle \bigoplus_{l+m=n} \bar k_lM \otimes \bar k_mN = \bar k_n(M \otimes N)$.
\end{thm}

In contrast to (b), the natural map
$$ R_s(k_nM) \ra k_n(R_sM)$$
need only be a monomorphism (with cokernel in $\Nil$).  See \exref{Rs ex}.

\subsection{Symmetric sequences of locally finite modules}

We construct a functor from $\U$ to the category of symmetric sequences of locally finite modules that seems to nicely encode much of the information about the Krull filtration of a module.

A {\em symmetric sequence} in $\U_0$ is a sequence $M = \{M_0,M_1,M_2, \dots\}$, with $M_n \in \Sigma_n\text{--}\U_0$.  The category of these, $\Sigma_*\text{--}\U_0$, has a symmetric monoidal structure with product
$$ (M \boxtimes N)_n = \bigoplus_{l+m=n} \Ind_{\Sigma_l \times \Sigma_m}^{\Sigma_n}(M_l \otimes N_m).$$

\begin{defn} Let $\sigma_*: \U \lra \Sigma_*\text{--}\U_0$ be defined by
$$ \sigma_nM = \bT^n k_nM.$$
\end{defn}

As $\bT$ is exact and $\bT^n k_{n-1}M = 0$,  $\sigma_n M$ also equals $\bT^n \bar k_nM$.  Thus, under the correspondence of \thmref{Un/Un-1 thm}, $\sigma_nM$ corresponds to the image of the $\U_{n-1}$--reduced composition factor $\bar k_nM$ in $\U_n/\U_{n-1}$.

\begin{thm} \label{sigma thm} $\sigma_*$ satisfies the following properties. \\

\noindent{\bf (a)} \ $\sigma_n$ is left exact and commutes with filtered colimits.\\

\noindent{\bf (b)} \ $\sigma_n$ commutes with the functors $\Sigma^s$, $\Phi$, $nil_s$, and $\bar R_s$. \\

\noindent{\bf (c)} \ $\sigma_n$ preserves $\Nil$--reduced modules, $\Nil$--closed modules, and \\ $\Nil$--isomorphisms.   \\

\noindent{\bf (d)} \ $\sigma_*$ is symmetric monoidal: $ \sigma_*(M \otimes N) = \sigma_*M \boxtimes \sigma_*N$. \\

\noindent{\bf (e)} \ For all $s$ and $n$, there is a natural isomorphism of $\Z/p[\Sigma_n]$--modules
$$(\sigma_nM)^s \simeq \Hom_{\U}(F(1)^{\otimes n}, \bar R_sM).$$

\end{thm}

\subsection{Organization of the paper}  The next three sections respectively contain needed background material on abelian categories, unstable modules, and polynomial functors.

\thmref{schwartz thm}, Schwartz's characterization of $\U_n$, is then proved in \secref{schwartz thm section}.

Properties of the Krull filtration of a module are proved in \secref{krull mod section}, and some of these are used in our proof of the identification of $\U_n/\U_{n-1}$ given in \secref{quotient cat section}.  In \secref{sym sequences section}, we discuss \thmref{sigma thm} and give examples illustrating it.

\subsection{Acknowledgements}

This research was partially funded by N.S.F. grant 0967649.  It was a pleasure to present these results at the Conference on Algebra and Topology celebrating the 60th birthday of Lionel Schwartz held in Nantes, in May, 2013.

\section{Background: abelian categories} \label{abelian categories}

We recall some standard concepts regarding abelian categories, as in \cite{gabriel}.

\subsection{Basic notions}

All the categories in this paper satisfy Grothendieck's axioms AB1--AB5 \cite{weibel} and have a set of generators.  Standard consequences include that such categories have enough injectives.

An object in an abelian category  $\C$ is {\em Noetherian} if its poset of subobjects satisfy the ascending chain condition.  $\C$ itself is said to be {\em locally Noetherian} if every object is the union of its Noetherian subobjects.  Standard consequences include that direct sums of injectives are again injective, and objects admit injective envelopes.

A full subcategory $\B$ of an abelian category $\C$ is {\em localizing} if it is closed under sub and quotient objects, extensions, and direct sums.

$f: M \ra N$ is a {\em $\B$-isomorphism} if $\ker f, \coker f \in \B$.

The quotient category $\C/\B$ has the same objects as $\C$, with morphisms from $M$ to $N$ given by equivalence classes of triples
$$ M \overset{f}{\hookleftarrow} M^{\prime} \xra{g} N^{\prime} \overset{h}{\twoheadleftarrow} N$$
with $\coker f, \ker h \in \B$.
It is the initial category under $\C$ in which all $\B$-isomorphisms have been inverted.

$M \in \C$ is {\em $\B$--reduced} if $\Hom_{\C}(N,M) = 0$ for all $N \in \B$, and is {\em $\B$--closed} if also $\Ext^1_{\C}(N,M) = 0$ for all $N \in \B$.

The exact quotient functor $l: \C \ra \C/\B$ has a right adjoint $r$.  The counit of the adjunction $\epsilon_M: lr(M) \ra M$ is always an isomorphism. The unit of the adjunction $\eta_M: M \ra rl(M)$ is thus idempotent, and is called {\em localization away from $\B$} or {\em $\B$--closure}.

The functor $k: \C \ra \B$ right adjoint to the inclusion $\B \hra \C$ can be computed by the formula $k(M) = \ker(\eta_M)$.

\subsection{Locally finite objects and the Krull filtration}

An object in an abelian category $\C$ is {\em simple} if it has no non-zero proper subobjects, {\em finite} if it has a filtration of finite length with simple composition factors, and {\em locally finite} if it is the a union of its finite subobjects.

We let $\C_0$ denote the full subcategory consisting of the locally finite objects of $\C$.  If $\C$ is locally Noetherian, $\C_0$ will be closed under extensions, and it follows that $\C_0$ is localizing.

\begin{defn} \cite[p.382]{gabriel} The {\em Krull filtration} of a locally Noetherian abelian category $\C$ is the increasing sequence of localizing subcategories
$$ \C_0 \subset \C_1 \subset \C_2 \subset $$
recursively defined for $n\geq 1$ by letting $\C_n$ be the full subcategory of $\C$ whose objects are the objects in $\C$ that represent locally finite objects in $\C/\C_{n-1}$.
\end{defn}

\begin{rem} Gabriel defines $\C_{\lambda}$ for any ordinal $\lambda$.  For example, $\C_{\omega}$ is defined as the smallest localizing category containing all the $\C_n$.  \thmref{F(1) thm} implies that $\U_{\omega} = \U$, so the Krull filtration for $\U$ stops at this point.
\end{rem}

\section{Background: unstable modules} \label{U section}

In this section we  recall some basic material about unstable modules.  A general reference for this is \cite{s2}.

\subsection{The categories $\U$ and $\K$}  The mod $p$ Steenrod algebra $\A$ is generated by $Sq^k$, $k\geq 0$, when $p=2$, and $P^k$, $k \geq 0$, together with the Bockstein $\beta$ when $p$ is odd, and satisfying the usual {Adem relations}.

The category $\U$ is then defined to be the full subcategory of $\A$--modules $M$ satisfying the {unstable condition}.  When $p=2$ this means that, for all $x\in M$, $Sq^k x = 0$ if $k> |x|$.  When $p$ is odd, the condition is that, for all $x\in M$ and $e = 0$ or $1$, $\beta^{e}P^k x = 0$ if $2k +e > |x|$.

The abelian category $\U$ has a tensor product coming from the {Cartan formula}, and $\K$ is then defined to be the category of commutative algebras $K$ in $\U$ also satisfying the {restriction condition}: $Sq^{|x|}x = x^2$ for all $x \in K$ when $p=2$, and $P^{|x|/2}x = x^p$ for all even degree $x \in K$ when $p$ is odd.

All of these definitions are, of course, motivated by the fact that, if $X$ is a topological space, its mod $p$ cohomology $H^*(X)$ is naturally an object in both $\U$ and $\K$.

\subsection{The Frobenius functor $\Phi$}

For $M \in \U$, $P_0: M \ra M$ is defined by
\begin{equation*}
P_0 x =
\begin{cases}
Sq^k x & \text{if } k=|x| \text{ and } p=2
\\ \beta^eP^kx = 0  & \text{if } 2k + e = |x| \text{ and $p$ is odd}.
\end{cases}
\end{equation*}

When $p=2$, the Frobenius functor $\Phi: \U \ra \U$ is defined by letting
\begin{equation*}
(\Phi M)^{m} =
\begin{cases}
M^n & \text{if } m=2n
\\ 0  & \text{otherwise, }
\end{cases}
\end{equation*}
with $Sq^{2k}\phi(x) = \phi(Sq^kx)$, where $\phi(x) \in (\Phi M)^{2n}$ corresponds to $x \in M^n$.

At odd primes, $\Phi: \U \ra \U$ is defined by letting
\begin{equation*}
(\Phi M)^{m} =
\begin{cases}
M^{2n+e} & \text{if } m=2pn+2e, \text{ with } e=0,1
\\ 0 & \text{otherwise,}
\end{cases}
\end{equation*}
with $P^{pk}\phi(x) = \phi(P^kx)$, and $P^{pk+1}\phi(x) = \phi(\beta P^kx)$ when $|x|$ is odd.

$\Phi$ is an exact functor, and there is a natural transformation of unstable modules
$$ \lambda: \Phi M\ra M$$
defined by $\lambda(\phi(x)) = P_0x$.

Let $\Omega: \U \ra \U$ be left adjoint to $\Sigma$. Explicitly $\Omega M$ is the largest unstable submodule of $\Sigma^{-1}M$.  This has just one nonzero right derived functor $\Omega^1$, and these can be calculated via an exact sequence
$$ 0 \ra \Sigma \Omega^1 M \ra \Phi(M) \xra{\lambda} M \ra \Sigma \Omega M \ra 0.$$

$\lambda: \Phi M\ra M$ is monic exactly when $M$ is $\Nil$--reduced, and, in this case, the evident iterated natural map $\lambda_k: \Phi^k M \ra M$ is still monic.

\subsection{The nilpotent filtration of a module}

Let $nil_s: \U \ra \Nil_s$ be right adjoint to the inclusion $\Nil_s \hra \U$. Explicitly, $nil_sM \subset M$ is the largest submodule of $M$ contained in $\Nil_s$.  We obtain a natural decreasing filtration of $M$:
$$ M = nil_0M \supset nil_1M \supset nil_2M \supset nil_3M \supset \dots .$$
This filtration is complete as $nil_s M$ is $(s-1)$--connected.

\begin{prop/def} $nil_sM/nil_{s+1}M = \Sigma^s R_sM$,
where $R_sM$ is $\Nil$-reduced.
\end{prop/def}

For a proof see \cite[Lemma 6.4.1]{s2} or \cite[Prop. 2.2]{k1}

We let $\bar R_sM$ denote the $\Nil$--closure of $R_sM$.

\begin{prop} \cite[Cor. 3.2]{k5} $\bar R_s: \U \ra \U$ is left exact.
\end{prop}

\begin{prop} \cite[Prop. 2.11]{k1} If $M$ is locally finite, then the nilpotent filtration equals the skeletal filtration, so that $M^s = R_s M = \bar R_s M$.
\end{prop}

\subsection{Finitely generated modules}

We need various results about unstable modules which are finitely generated over $\A$.

\begin{thm} \label{finite gen thm} {\bf (a)} \   A submodule of a finitely generated unstable module is again finitely generated.  Thus $\U$ is locally Noetherian.

\noindent{\bf (b)} \ The nilpotent filtration of a finitely generated module has finite length.

\noindent{\bf (c)} \ A finitely generated unstable module represents a finite object in $\U/\Nil$.

\end{thm}

Proofs of these appear in \cite{s2} and \cite{k1}.

\begin{cor} \label{gen cor} Any localizing subcategory of $\U$ will be generated by modules of the form $\Sigma^s M$, with $M$ finitely generated, $\Nil$--reduced, and representing a simple object in $U/\Nil$.
\end{cor}

We will also need the following lemma.

\begin{lem} \label{phi lemma} Let $M$ be $\Nil$--reduced and finitely generated, and let $i: N \hra M$ be the inclusion of a submodule with $M/N \in \Nil$.  Then, for $k\gg 0$, the inclusion $\lambda_k: \Phi^k M \hra M$ factors through $i$.
\end{lem}
\begin{proof}  Let $x_1, \dots, x_l$ generate $M$.  Since $M/N \in \Nil$, there exists $k$ such that $P_0^k(x_i) \in N$ for all $i$.  But then the image of $\lambda^k$ will be contained in $N$.
\end{proof}

\subsection{$\U$--injectives and properties of $\bar T$}

Let $J(n)$ be the $n$th `Brown--Gitler module': the finite injective representing $M \rightsquigarrow (M^n)^{\vee}$ \cite[\S2.3]{s2}.

\begin{thm}  $H^*(BV) \otimes J(n)$ is injective in $\U$, and every $\U$--injective is a direct summand of a direct sum of such modules.
\end{thm}

\begin{thm}  {\bf (a)}  $T$, and thus $\bar T$, is exact.

\noindent{\bf (b)}  The natural map $T(M \otimes N) \ra TM \otimes TN$ is an isomorphism.

\noindent{\bf (c)}  $T$, and thus $\bar T$, preserves both $\Nil$--reduced and $\Nil$--closed modules.

\noindent{\bf (d)} $T$, and thus $\bar T$, commutes with the following functors: $\Sigma^s$, $\Phi$, $nil_s$, $R_s$, and $\bar R_s$.
\end{thm}

All of this in the literature: see \cite{L}, \cite{ls2}, \cite{lz1}, \cite{lz2}, \cite{k1}.

\section{Background: polynomial functors}

\subsection{Definitions and examples} Recall that $\F$ is the category of functors from finite dimensional $\Z/p$--vector spaces to $\Z/p$--vector spaces.  This is an abelian category in the standard way, e.g., $F \ra G \ra H$ is exact at $G$ means that $F(V) \ra G(V) \ra H(V)$ is exact at $G(V)$ for all $V$.

Some objects in $\F$ are $S^n$, $H_n$, $P_W$, and $I_W$,  defined by $S^n(V) = (V^{\otimes n})_{\Sigma_n}$, $H_n(V) = H_n(BV)$, $P_W(V) = \Z/p[\Hom(W,V)]$ and $I_W(V) = \Z/p^{\Hom(V,W)}$.  Note that $H_0$ is the constant functor `$\Z/p$', and $H_1$ is the identity functor $\Id$.

Using Yoneda's lemma, one see that $\Hom_{\F}(P_W,F) \simeq F(W)$, and thus $P_W$ is projective.  Similarly $\Hom_{\F}(F, I_W) \simeq F(W)^{\vee}$, and thus $I_W$ is injective.  Note also that $P_V \otimes P_W \simeq P_{V \oplus W}$ and $I_V \otimes I_W \simeq I_{V \oplus W}$.  The functors $P_W$ and $I_W$ canonically split as $P_W \simeq \Z/p \oplus \bar P_W$ and $I_W \simeq \Z/p \oplus \bar I_W$, and we write $\bar P_{\Z/p}$ and $\bar I_{\Z/p}$ as $\bar P$ and $\bar I$.

$F(V)$ is a canonical retract of $F(V \oplus \Z/p)$ and, as in the introduction, one defines the exact functor $\Delta: \F \ra \F$ by
$$ (\Delta F)(V) = F(V \oplus \Z/p)/F(V).$$
One easily checks that $\Delta$ has a left adjoint given by $F \mapsto F \otimes \bar P$, and a right adjoint given by $F \mapsto F \otimes \bar I$.

$F$ is {\em polynomial of degree $n$} if $\Delta^{n+1}F = 0$.  As explained in \cite{genrep1} or \cite{s2}, this agrees with the Eilenberg--MacLane definition used in \cite{hls1}.  One then lets $\F^n$ be the category of all degree $n$ polynomial functors and $\F^{an}$ be the category of all locally polynomial functors. (As shown in \cite{genrep1}, $\F^{an}$ is also the locally finite category $\F_0 \subset \F$.) As examples, $\Id^{\otimes n}, S^n, H_n \in \F^n$, $\bar P_W$ has no nonzero polynomial subfunctors, while $I_W \in \F^{an}$.

We explain this last fact.  There is an identification
$$ I_{\Z/p}(V) = S^*(V)/(x^p-x)$$
and thus $I_W$ is visibly a quotient of the sum of the polynomial functors
$$ V \mapsto S^n(\Hom(W,V)).$$

\subsection{The polynomial filtration of a functor}
The inclusion $\F^n \hra \F$ has both a left adjoint $q_n$ and a right adjoint $p_n$.  Explicitly
$$ q_nF = \coker \{ \epsilon_F: \Delta^{n+1}F \otimes \bar P^{\otimes (n+1)} \ra F\}$$
and
$$ p_nF = \ker \{ \eta_F: F \ra \Delta^{n+1}F \otimes \bar I^{\otimes (n+1)} \}.$$

\begin{thm} \label{poly filt thm} {\bf (a)} \ $\displaystyle p_nI_W(V) = S^{* \leq n}(\Hom(W,V))/(x^p-x)$. \\

\noindent{\bf (b)} \ $\displaystyle \sum_{l+m=n} p_lF \otimes p_mG = p_n(F \otimes G)$.
\end{thm}

Both of these statements are well known.  For (a), see \cite[Lemma 7.6.6]{hls1} or the discussion in \cite[\S 6]{genrepsurvey}.  Statement (b) is then a consequence as follows.  We are asserting that two natural filtrations on $F \otimes G$ agree, with the one filtration clearly including in the other.  From (a), one can visibly see that these filtrations agree when $F$ and $G$ are sums of $I_W$'s.  For the general case, one finds exact sequences
$$ 0 \ra F \ra I_0 \ra I_1$$
and
$$ 0 \ra G \ra J_0 \ra J_1,$$
where $I_0$, $I_1$, $J_0$, and $J_1$ are all sums of $I_W$'s.  Tensoring these sequences together yields an exact sequence
$$ 0 \ra F \otimes G \ra I_0 \otimes J_0 \ra I_1 \otimes J_0 \oplus I_0 \otimes J_1.$$
The two filtrations agree on the last two terms, and thus on $F\otimes G$.

\subsection{The equivalence $\U/\Nil \simeq \F^{an}$} \label{U/Nil subsection}

One has adjoint functors
$$ \U \begin{array}{c} l \\[-.08in] \longrightarrow \\[-.1in] \longleftarrow \\[-.1in] r
\end{array} \F$$
defined by $l(M)(V) = \Hom_{\U}(M,H^*(BV))^{\vee}$ and $r(F)^n = \Hom_{\F}(H_n,F)$.  As examples, $l(F(n)) = H_n$, and $r(I_W) = H^*(BW)$.

\begin{thm} \cite{hls1} The functor $l$ is exact, and $l$ and $r$ induce an equivalence of abelian categories
$$ \U/\Nil \begin{array}{c} l \\[-.08in] \longrightarrow \\[-.1in] \longleftarrow \\[-.1in] r
\end{array} \F^{an}.$$
\end{thm}

It follows that $M \ra rlM$ is $\Nil$--closure.

\begin{prop} \cite{hls1} There are natural isomorphisms
$$ l(M \otimes N) \simeq l(M) \otimes l(N) \text{ and } r(F \otimes G) \simeq r(F) \otimes r(G).$$
\end{prop}

\begin{prop}  There are natural isomorphisms
$$ l (\bar T M) \simeq \Delta l(M) \text{ and } r(\Delta F) \simeq \bar T r(F).$$
\end{prop}

The first of these is easily checked, and the second formally follows, once one knows that $\bar T$ preserves $\Nil$--closed modules.

\section{Proof of \thmref{schwartz thm}} \label{schwartz thm section}

In this section, we prove Schwartz' characterization of $\U_n$:
$$\U_n = \U_n^T,$$
where $\U_n^T$ is the full subcategory of $\U$ with objects $\{ M \in \U \ | \ \bT^{n+1}M = 0 \}$.

We prove this by induction on $n$.  As our inductive step needs a lemma (\corref{deg 0 cor} below) that is at the heart of the $n=0$ case, we begin with the proof that $\U_0 = \U_0^T$ roughly following the proofs in \cite{ls2,s2}.

It is easy to characterize modules in $\U_0$, noting that any module has its decreasing skeletal filtration, with subquotient modules concentrated in one degree.

\begin{lem} \  Let $M$ be an unstable module. $M$ is simple if and only if it is isomorphic to $\Sigma^s \Z/p$ for some $s$.  $M$ is finite in the categorical sense if and only if it is finite.  $M$ is locally finite if and only if $\A\cdot x \subseteq M$ is finite for all $x \in M$.
\end{lem}

\begin{prop}  If $M$ is locally finite, then $\bar T M = 0$.  Thus $\U_0 \subseteq \U_0^T$.
\end{prop}
\begin{proof}  As $\bar T$ is exact and commutes with suspensions and directed colimits, this follows from the lemma and the calculation that $\bar T \Z/p = 0$.
\end{proof}
\begin{prop} If $\bar T M = 0$, then $M$ is locally finite. Thus $\U_0^T \subseteq \U_0$.
\end{prop}
\begin{proof}[Sketch proof]  It suffices to show this when $M$ is finitely generated.  Such an $M$ embeds in an injective module of the form $\displaystyle \bigoplus_{i=1}^r H^*(BV_i) \otimes J(n_i)$.

Meanwhile, $\bar T M = 0$ implies that $\Hom_{\U}(M, \tilde H^*(BV) \otimes J(n)) = 0$ for all $V$ and $n$.  Thus $M$ will embed in $\displaystyle \bigoplus_{i=1}^r H^0(BV_i) \otimes J(n_i) = \bigoplus_{i=1}^r J(n_i)$ and so is finite.
\end{proof}

\begin{cor} \label{deg 0 cor} If $M$ is $\Nil$--reduced and $\bar T M = 0$, then $M$ is concentrated in degree 0.
\end{cor}

This corollary is used in the last step of the proof of the following lemma.

\begin{lem}  \label{omega lemma} If $M$ is $\Nil$--reduced, then
$$\bar T^{n+1}M = 0 \Leftrightarrow \bar T^n\Omega M = 0.$$
\end{lem}
\begin{proof}  If $M$ is $\Nil$--reduced, there is a short exact sequence
$$ 0 \ra \Phi M \ra M \ra \Sigma \Omega M \ra 0.$$
Since $\bar T^n$ is exact and commutes with $\Phi$ and $\Sigma$, one deduces that
\begin{equation*}
\begin{split}
\bar T^n \Omega M = 0 &
\Leftrightarrow\Phi \bar T^n M \simeq \bar T^n M \\
  & \Leftrightarrow \bar T^n M \text{ is concentrated in degree 0} \\
  & \Leftrightarrow \bar T^{n+1} M = 0.
\end{split}
\end{equation*}
\end{proof}

Armed with this lemma, we now give the inductive step of the proof that $\U_n = \U_n^T$.  So assume by induction that $\U_{n-1} = \U_{n-1}^T$.

\begin{proof}[Proof that $\U_n \subseteq \U_n^T$]  A simple object in $\U/\U_{n-1}$ can be represented by $\Sigma^s M$, where $M$ is reduced.  We show that then $\Sigma^s M \in \U_n^T$.

Consider the exact sequence
$$ 0 \ra \Sigma^s \Phi M \ra \Sigma^s M \ra \Sigma^{s+1} \Omega M \ra 0.$$
As $\Sigma^s M$ is simple in $\U/\U_{n-1}$ either $\Sigma^s \Phi M \in \U_{n-1}$ or $\Sigma^{s+1} \Omega M \in \U_{n-1}$.

In this first case, we have
\begin{equation*}
\begin{split}
\Sigma^s \Phi M \in \U_{n-1} & \Rightarrow \bar T^n \Sigma^s \Phi M = 0  \Rightarrow \Phi \bar T^n M = 0 \\
& \Rightarrow \bar T^n M = 0  \Rightarrow \bar T^n \Sigma^s M = 0 \\
& \Rightarrow \Sigma^s M \in \U_{n-1}.
\end{split}
\end{equation*}
But this contradicts that $\Sigma^sM$ is nonzero in $\U/\U_{n-1}$.

Thus $\Sigma^{s+1} \Omega M \in \U_{n-1}$.  But then
\begin{equation*}
\begin{split}
\Sigma^{s+1} \Omega M \in \U_{n-1} & \Rightarrow \bar T^n \Sigma^{s+1} \Omega M = 0  \Rightarrow \bar T^n \Omega M = 0 \\
& \Rightarrow \bar T^{n+1} M = 0  \Rightarrow \bar T^{n+1} \Sigma^s M = 0 \\
& \Rightarrow \Sigma^s M \in \U_{n}^T.
\end{split}
\end{equation*}
\end{proof}

\begin{proof}[Proof that $\U_n^T \subseteq \U_n$]  Suppose that $\Sigma^s M \in \U_n^T$, with $M$ finitely generated, $\Nil$--reduced, and representing a simple object in $\U/\Nil$.  We show that then $\Sigma^s M$ is either in $\U_{n-1}$ or represents a simple object in $\U/\U_{n-1}$, and is thus in $\U_n$.

Using \lemref{omega lemma}, $\Sigma^s M \in \U_n^T$ implies that $\Omega M \in \U_{n-1}^T$, so that $\Phi M \ra M$ is a $\U_{n-1}^T$--isomorphism.  It follows that, for all $k$, $\Sigma^s \Phi^k M \ra \Sigma^s M$ will be a $\U_{n-1}^T$--isomorphism (and thus a $\U_{n-1}$--isomorphism).

Any nonzero sub-object of $\Sigma^s M$ has the form $\Sigma^s N$ with $0 \neq N < M$.  As $M$ is $\Nil$--reduced and simple in $\U/\Nil$, $M/N \in \Nil$.  By \lemref{phi lemma}, there exists $k>0$ such that $\Phi^k M \subset N$.  Since $\Sigma^s M/\Sigma^s \Phi^kM \in \U_{n-1}$, we deduce that $\Sigma^s M/\Sigma^s N \in \U_{n-1}$.
\end{proof}

\begin{rem}  We comment on how the proof of \thmref{schwartz thm} differs from the one outlined in \cite{s2}.

The proof in \cite{s2} makes use of the theorem that $\U/\Nil \simeq \F^{an}$ together with substantial analysis of how the polynomial filtration of $\F^{an}$ is reflected in the category of reduced modules.

Our proof instead uses \thmref{finite gen thm}(c), which says that finitely generated modules represent finite objects in $\U/\Nil$.  The proof of this also uses that $\U/\Nil \simeq \F^{an}$, but replaces the analysis of the polynomial filtration by the observation that the functor in $\F$ given by $V \mapsto H_n(BV)$ is a finite functor, which has a very elementary proof \cite[Prop. 4.10]{genrep1}.
\end{rem}

\section{Properties of the Krull filtration of a module} \label{krull mod section}

For simplicity, we write $\bar H$ denote $\tilde H^*(B\Z/p)$.  Recall that $M \in \U_n$ if and only if $\bar T^{n+1}M = 0$, and $\bar T^{n+1}$ is both exact and left adjoint to the functor $M \mapsto M \otimes \bar H^{\otimes n+1}$. It follows formally that
\begin{equation*} \label{kn defn}
k_nM = \ker\{ \eta_M: M \ra \bar T^{n+1} M \otimes \bar H^{\otimes n+1}\},
\end{equation*}
where $\eta_M$ is the unit of the adjunction.

We run through proofs of various properties of $k_nM$. \\

\noindent{\bf (a)} \  $k_n$ is left exact and commutes with filtered colimits.
\begin{proof}  This is immediate.
\end{proof}

\noindent{\bf (b)} \  $\displaystyle \bigcup_{n=0}^{\infty} k_nM = M$.
\begin{proof}  One easily computes that $\bar T F(n) = F(n-1)$, so that $F(n) \in \U_n$.  As every unstable module is a quotient of a sum of $F(n)$'s, the claim follows.
\end{proof}

\noindent{\bf (c)} \  $k_n(N \otimes M) \simeq N \otimes k_nM$ if $N \in \U_0$. In particular, $k_n\Sigma^s M = \Sigma^s k_nM$.
\begin{proof}  As $\bar TN=0$,
$$\eta_{N \otimes M}: N \otimes M \ra \bar T^{n+1}(N \otimes M) \otimes \bar H^{\otimes n+1}$$ identifies with
$$ N \otimes \eta_M: N \otimes M \ra N \otimes \bar T^{n+1}M \otimes \bar H^{\otimes n+1}.$$
\end{proof}

\noindent{\bf (d)} \  $k_n\Phi M \simeq \Phi k_nM$.
\begin{proof}  As $\Phi$ commutes with tensor products and $\bar T$, one has a commutative diagram
\begin{equation*}
\xymatrix{
\Phi M \ar@{=}[d] \ar[r]^-{\Phi(\eta_M)} & \Phi (\bar T^{n+1} M \otimes \bar  H^{\otimes n+1})\ar[r]^-{\sim} & \bar T^{n+1} \Phi M \otimes (\Phi \bar  H)^{\otimes n+1} \ar[d]  \\
\Phi M \ar[rr]^-{\eta_{\Phi M}} && \bar T^{n+1} \Phi M \otimes \bar H^{\otimes n+1} }
\end{equation*}
As $\Phi \bar H \ra \bar H$ is monic, the kernels of the two horizontal maps are equal. The kernel of the bottom map is $k_n\Phi M$, and, since $\Phi$ is exact, the kernel of the top map is $\Phi k_n M$.
\end{proof}

\noindent{\bf (e)} \  $k_n nil_s M \simeq nil_s k_nM$.
\begin{proof}  One easily checks that $k_n nil_s M \simeq k_nM \cap nil_s M \simeq nil_s k_nM$.
\end{proof}

\noindent{\bf (f)} \  If $M = r(F)$, then $k_n M = r(p_n F)$.  Thus if $M$ is $\Nil$--closed, so is $k_nM$.

\begin{proof}  Applying the left exact functor $r$ to the exact sequence
$$ 0 \ra p_n F \ra F \ra \Delta^{n+1} F \otimes \bar I^{\otimes n+1}$$
shows that $r(p_n F)$ is the kernel of the map
$$ M \ra r(\Delta^{n+1} F \otimes \bar I^{\otimes n+1}).$$
But since $r$ commutes with tensor products and $r \circ \Delta^{n+1} = \bar T^{n+1} \circ r$, this map rewrites as
$$ M \ra \bar T^{n+1} M \otimes \bar H^{\otimes n+1},$$
which has kernel $k_n M$.
\end{proof}

\noindent{\bf (g)} \  $k_n$ preserves $\Nil$--reduced modules.
\begin{proof}
This follows from (f), noting that $\Nil$--reduced modules are submodules of $\Nil$--closed modules, and $k_n$ preserves monomorphisms.
\end{proof}

\noindent{\bf (h)} \  If $F = l(M)$, then $p_nF = l(k_nM)$.  It follows that $k_n$ preserves $\Nil$--isomorphisms.
\begin{proof}  Applying the exact functor $l$ to the exact sequence
$$ 0 \ra k_n M \ra M \ra \bar T^{n+1} M \otimes \bar H^{\otimes n+1}$$
shows that $l(k_n M)$ is the kernel of the map
$$ F \ra l(\bar T^{n+1} M \otimes \bar H^{\otimes n+1}).$$
But since $l$ commutes with tensor products and $\Delta^{n+1} \circ l= l \circ \bar T^{n+1}$, this map rewrites as
$$ F \ra \Delta^{n+1} F \otimes \bar I^{\otimes n+1},$$
which has kernel $p_n F$.

A map $f: M \ra N$ in $\U$ is a $\Nil$--isomorphism precisely when $l(f)$ is an isomorphism.  When this happens, $l(k_nf)=p_nl(f)$ will be an isomorphism, so $k_nf$ will be a $\Nil$--isomorphism.
\end{proof}

\noindent{\bf (i)} \  The natural map $\bar R_sk_nM \ra k_n\bar R_sM$
is an isomorphism.
\begin{proof}
Applying the left exact functor $\bar R_s$ to the exact sequence
$$ 0 \ra k_nM \ra  M \ra \bar T^{n+1} M \otimes \bar H^{\otimes n+1}$$
shows that $\bar R_sk_nM$ is the kernel of the map
$$ \bar R_sM \ra \bar R_s(\bar T^{n+1} M \otimes \bar H^{\otimes n+1}).$$
But since $\bar H$ is $\Nil$--closed, and $\bar R_s$ commutes with $\bar T$, this map rewrites as
$$ R_sM \ra \bar T^{n+1} \bar R_s M \otimes \bar H^{\otimes n+1},$$
which has kernel $k_n\bar R_sM$.
\end{proof}

\noindent{\bf (j)} \ $\displaystyle \sum_{l+m=n} k_lM \otimes k_mM = k_n(M \otimes N)$.
\begin{proof}  As with \thmref{poly filt thm}, this says that two natural filtrations of $M \otimes N$ agree, with one filtration certainly including in the other.  Combined with (f), \thmref{poly filt thm} says that this is true if both modules are $\Nil$--closed.

Then (c) implies that the statement is also true if both $M$ and $N$ are sums of $\Nil$--closed modules tensored with locally finite modules.  This includes all $\U$--injectives.

The statement then holds for all modules, using the same argument as in the proof of \thmref{poly filt thm}.
\end{proof}

From the above, one can deduce that the natural map $R_sk_nM \ra k_nR_sM$ is always an inclusion with cokernel in $\Nil$, but the next example shows that this need {\em not} be an isomorphism.

\begin{ex}  \label{Rs ex} With $p=2$, let $M \in \U$ be defined by the pullback square
\begin{equation*}
\xymatrix{
 M \ar[d] \ar[r] & \Phi^2(F(1)) \ar@{->>}[d]  \\
\Sigma F(3) \ar@{->>}[r] & \Sigma^4 \Z/2. }
\end{equation*}
We claim that, for this $M$, the natural monomorphism $R_0(k_1(M)) \ra k_1(R_0(M))$ identifies with the proper inclusion
$ \Phi^3(F(1)) \hra \Phi^2(F(1))$, and is thus {\em not} an isomorphism.

To see this, recall that $R_0M = M/nil_1 M$.  Applying $R_0$ to the short exact sequence
$$0 \ra \Sigma(F(3)^{\geq 4}) \ra M \ra \Phi^2F(1) \ra 0$$
shows that $R_0M$, and thus $k_1R_0M$, equals $\Phi^2F(1)$.

Meanwhile, applying the left exact functor $k_1$ to the short exact sequence
$$ 0 \ra \Phi^3F(1) \ra M \ra \Sigma F(3) \ra 0$$
shows that $k_1M$, and thus $R_0k_1M$, equals $\Phi^3F(1)$.
\end{ex}

\begin{rem} \label{no adjoint remark} Being a right adjoint, $k_n$ is left exact.  It is easily seen to not usually preserve surjections.  To see this, let $M^{>r} \subset M$ denote the submodule of all elements of degree greater than $r$.  Then $F(1) \ra F(1)/F(1)^{>1} = \Sigma \Z/p$ is a surjection, $k_0(F(1)) = 0$, but $k_0(F(1)/F(1)^{>1}) = \Sigma \Z/p$.

A similar argument shows that $\U_n \hra \U$ does not admit a left adjoint. So see this, note that, for all $M \in \U$ and all $r$, $M/M^{>r}$ is in $\U_0$.  Now suppose a left adjoint $q_n: \U \ra \U_n$ exists for some $n$.  It would follow that $M \ra M/M^{>r}$ would factor through the natural map $M \ra q_n(M)$ for all $r$, and we could deduce that $M \ra q_n(M)$ would be monic.  Since this would be true for all $M \in \U$, it would follow that $\U_n = \U$.
\end{rem}

We end this section with a discussion of  \propref{pnH prop}: $k_n \bar H$ is the span of products of elements in $F(1)$ of length at most $n$.

This is essentially the content of \cite[Lem.7.6.6]{hls1}, and a proof goes along the following lines.  $I = I_{\Z/p}$ is a Hopf algebra object in $\F$, with addition $\Z/p \oplus \Z/p \ra \Z/p$ inducing the coproduct
$$ \Psi: I_{\Z/p} \ra I_{\Z/p \oplus \Z/p} \simeq I_{\Z/p} \otimes I_{\Z/p}.$$
Then one observes that $p_n\bar I$ is precisely the kernel of the iterated reduced coproduct
$$ \Psi^n: \bar I \ra \bar I^{\otimes n+1}.$$

Applying $r$ to this, it follows that $k_n\bar H$ is the kernel of the iterated reduced coproduct
$$ \Psi^n: \bar H \ra \bar H^{\otimes n+1},$$
i.e., is the $n$th stage of the primitive filtration of $H^*(B\Z/p)$.  This is then checked to be the span of products of elements in $F(1)$ of length at most $n$.

\section{The identification of $\U_n/\U_{n-1}$} \label{quotient cat section}

\thmref{Un/Un-1 thm} is a consequence of the next two lemmas.

\begin{lem} \label{adjoint lem} The functor $\Sigma_n\text{--}\U_0 \ra \U_n$ given by
$$N \mapsto (N \otimes F(1)^{\otimes n})^{\Sigma_n}$$
is right adjoint to $\bT^n: \U_n \ra \Sigma_n\text{--}\U_0$.
\end{lem}

\begin{lem} \label{counit lem}  The counit of the adjunction
$$ \bT^n((N \otimes F(1)^{\otimes n})^{\Sigma_n}) \ra N$$
is an isomorphism for all $N \in \Sigma_n\text{--}\U_0$.
\end{lem}

\begin{proof}[Proof of \lemref{adjoint lem}]  Given $M \in \U_n$ and $N \in \Sigma_n\text{--}\U_0$, we compute:
\begin{equation*}
\begin{split}
\Hom_{\Sigma_n\text{--}\U_0}(\bT^n M, N)
 &  = \Hom_{\U}^{\Sigma_n}(\bT^n M, N) \\
 &  = \Hom_{\U}^{\Sigma_n}(M, N \otimes \bar H^{\otimes n}) \\
 &  = \Hom_{\U}^{\Sigma_n}(M, k_n(N \otimes \bar H^{\otimes n})) \text{\  (since } M \in \U_n) \\
 &  = \Hom_{\U}^{\Sigma_n}(M, N \otimes F(1)^{\otimes n})  \\
 &  = \Hom_{\U_n}(M, (N \otimes F(1)^{\otimes n})^{\Sigma_n}).  \\
\end{split}
\end{equation*}
\end{proof}

\begin{proof}[Proof of \lemref{counit lem}]
Starting from the calculation that $\bT F(1) \simeq \Z/p$, it is easy to check that $\bT^n (F(1)^{\otimes n}) = \Z/p[\Sigma_n]$.

Then one has isomorphisms
\begin{equation*}
\begin{split}
\bT^n((N \otimes F(1)^{\otimes n})^{\Sigma_n})
 & \simeq (\bT^n(N \otimes F(1)^{\otimes n}))^{\Sigma_n} \\
 & \simeq (N \otimes \bT^n(F(1)^{\otimes n}))^{\Sigma_n} \\
 & \simeq (N \otimes \Z/p[\Sigma_n])^{\Sigma_n} \\
 & \simeq N.
\end{split}
\end{equation*}
\end{proof}

\begin{proof}[Proof of \thmref{recursive Un thm}]
Given $M \in \U$, suppose that there exist $K,Q \in \U_{n-1}$, $N \in \Sigma_n\text{--}\U_0$, and an exact sequence
$$ 0 \ra K \ra M \ra (N \otimes F(1)^{\otimes n})^{\Sigma_n} \ra Q \ra 0.$$
Applying the exact functor $\bT^n$ to this sequence shows that $\bT^n M \simeq N$, and then applying $\bT$ once more shows that $\bT^{n+1}M = 0$, so that $M \in \U_n$.

Conversely, given $M \in \U_n$, \thmref{Un/Un-1 thm} tells us that the unit of the adjunction
$$ M \ra (\bT^n M \otimes F(1)^{\otimes})^{\Sigma_n}$$
has kernel and cokernel in $\U_{n-1}$.
\end{proof}

\begin{ex}  \label{F(n) example} $\U_n$ is generally not generated as a localizing category by the modules $\Sigma^s F(m)$ with $m \leq n$.  If it were, then $\F^n$ would be generated as a localizing category by the functors $H_m$, for $m\leq n$.   But this is not always true.

The first example of this is $\F^3$, when $p=2$.  There is a splitting of functors
$$ \Lambda^2(V) \otimes V \simeq \Lambda^3(V) \oplus L(V).$$
$L$ is a simple object in $\F^3$ which is not among the composition factors of the functors $H_m$, $m\leq 3$: these are just the functors $\Lambda^m$, $m \leq 3$.
\end{ex}

\section{A functor to symmetric sequences in $\U_0$} \label{sym sequences section}

Recall that $\Sigma_*\text{--}\U_0$ is the category of symmetric sequences of locally finite modules, and that the functor
$$\sigma_*: \U \lra \Sigma_*\text{--}\U_0$$
is defined by $\sigma_nM = \bT^n k_nM$.

We prove the various properties of this functor asserted in \thmref{sigma thm}.

Properties listed in parts (a)--(c) of the theorem are true because the analogous properties have been shown to be true for the functors $\bT$ and $k_n$. \\

\noindent{\bf (d)} \ $\sigma_*$ is symmetric monoidal: $ \sigma_*(M \otimes N) = \sigma_*M \boxtimes \sigma_*N$.

\begin{proof}  Recall that $\sigma_nM = \bT^n \bar k_nM$, where $\bar k_nM = k_nM/k_{n-1}M$.  Then we compute:
\begin{equation*}
\begin{split}
\sigma_n(M\otimes N)
 & = \bT^n\bar k_n(M\otimes N) \\
 & =\bigoplus_{l+m=n}\bT^n(\bar k_lM \otimes \bar k_mN) \\
 & =\bigoplus_{l+m=n} \Ind_{\Sigma_l \times \Sigma_m}^{\Sigma_n}(\bT^l\bar k_lM \otimes \bT^m\bar k_mN) \\
 & = (\sigma_*M \boxtimes \sigma_*N)_n.
\end{split}
\end{equation*}
\end{proof}

\noindent{\bf (e)} \ For all $s$ and $n$, there is a natural isomorphism of $\Z/p[\Sigma_n]$--modules
$$(\sigma_nM)^s \simeq \Hom_{\U}(F(1)^{\otimes n}, \bar R_sM).$$
\begin{proof}

We first look at the special case when $M$ is $\Nil$--closed.  We claim that then $\sigma_n M$ is concentrated in degree 0, and
$$(\sigma_nM)^0 \simeq \Hom_{\U}(F(1)^{\otimes n}, M).$$

To see this, note that $M = r(F)$ for some $F \in \F$.  Thus
$$ \sigma_n M = \bT^n k_n r(F)= r(\Delta^n p_n F).$$
Since $p_nF$ is polynomial of degree $n$, $\Delta^n p_nF$ will be polynomial of degree 0, and is thus constant.  It follows that
\begin{multline*}
r(\Delta^n p_n F)
= \Hom_{\F}(\Z/p, \Delta^n p_n F)  = \Hom_{\F}(\bar P^{\otimes n}, p_n F) \\
  = \Hom_{\F}(q_n(\bar P^{\otimes n}), F) = \Hom_{\F}(\Id^{\otimes n}, F)  = \Hom_{\U}(F(1)^{\otimes n}, M).
\end{multline*}

Now we consider the case of a general unstable module $M$.  For locally finite modules, the skeletal filtration equals the nilpotent filtration, so
$$ (\sigma_nM)^s = \bar R_s \sigma_n M = \sigma_n \bar R_s M = \Hom_{\U}(F(1)^{\otimes n}, \bar R_s M).$$
\end{proof}

We make a couple of general comments regarding the values of $\sigma_*M$.

Firstly, from part (e) of \thmref{sigma thm}, one sees that if the nilpotence length of $M$ is finite, there will be uniform bound on the nonzero degrees of the modules $\sigma_nM$.  More precisely, if $nil_{s+1}M = 0$, then $(\sigma_nM)^t = 0$ for all $t>s$ and all $n$. This is the case if $M$ is a module over a Noetherian unstable algebra $K$ which is finitely generated using both the $K$--module and $\A$--module structures \cite[Lemma 6.3.10]{meyer}. (The more accessible paper \cite{henn} has a weaker version, and both \cite{hls2} and \cite{k5} give examples of bounds on nilpotence.)

Our second comment is that if $K$ is an unstable algebra, then $\sigma_*K$ will be an  algebra in symmetric sequences.  Explicitly, the multiplication on $K$ induces $\Sigma_i \times \Sigma_j$ equivariant maps $\sigma_iK \otimes \sigma_jK \ra \sigma_{i+j}K$ which are suitably associative, commutative, and unital.  In particular, each $\sigma_nK$ will be a module over the unstable algebra $\sigma_0K$, the locally finite part of $K$.

We end this section with various examples, all restricted to $p=2$.

\begin{prop} Suppose $M$ is `homogeneous of degree $m$', i.e., $\bar k_m M = M$.  Then $\bT^m M \in \Sigma_m\text{--}\U_0$, and
$$
\sigma_nM =
\begin{cases}
\bT^mM  & \text{if } n=m \\ 0 & \text{otherwise}.
\end{cases}
$$
\end{prop}
\begin{proof}  The hypothesis is equivalent to saying that the unit of the adjunction $M \ra (\bT^m M \otimes F(1)^{\otimes m})^{\Sigma_m}$ is an embedding, and the result follows from \thmref{Un/Un-1 thm}.
\end{proof}

\begin{ex}  As special cases of the proposition, one has
$$
\sigma_nF(m) =
\begin{cases}
\Z/2  & \text{if } n=m \\ 0 & \text{otherwise},
\end{cases}
$$
and
$$
\sigma_nF(1)^{\otimes m} =
\begin{cases}
\Z/2[\Sigma_m] & \text{if } n=m \\ 0 & \text{otherwise}.
\end{cases}
$$
\end{ex}

Let $U: \U \ra \K$ be the usual free functor, left adjoint to the forgetful functor.  The next proposition describes $\sigma_*K$ when $K= U(M)$ where $M$ is as in the last example.

To state this, we define a functor
$$Sh^m: \Sigma_m\text{--}\U_0 \ra \text{algebras in } \Sigma_*\text{--}\U_0$$  as follows.  Given a $\Sigma_m$--module $N$, let
$$ Sh^m_n(N) =
\begin{cases}
\Ind_{\Sigma_k \wr \Sigma_m}^{\Sigma_{km}}(N^{\otimes k}) & \text{if } n=km \\ 0 & \text{otherwise}.
\end{cases}
$$
Given $i+j=k$, the multiplication map $ Sh^m_{im}(N) \otimes Sh^m_{jm}(N) \ra Sh^m_{km}(N)$
is defined by the composite
$$\Ind_{\Sigma_i \wr \Sigma_m}^{\Sigma_{im}} N^{\otimes i} \otimes \Ind_{\Sigma_j \wr \Sigma_m}^{\Sigma_{m}} N^{\otimes j} \ra
\Ind_{(\Sigma_i \times \Sigma_j) \wr \Sigma_m}^{\Sigma_{im} \times \Sigma_{jm}} N^{\otimes k} \ra
\Ind_{(\Sigma_k) \wr \Sigma_m}^{\Sigma_{km}} N^{\otimes k},$$
where the first map is $\Ind_{(\Sigma_i \times \Sigma_j) \wr \Sigma_m}^{\Sigma_{im} \times \Sigma_{jm}}$ applied to the shuffle product
$$ N^{\otimes i} \otimes N^{\otimes j} \ra N^{\otimes k}.$$

\begin{prop}  Suppose $M$ satisfies $\bar k_mM = M$.  Then
$ \sigma_*(U(M)) = Sh^m_*(\bT^mM)$.
\end{prop}

This follows from the next two lemmas.

\begin{lem} Suppose $M$ satisfies $\bar k_mM = M$.  Then
\begin{equation*}
\bar k_nU(M) =
\begin{cases}
\Lambda^k M & \text{if } n = km \\ 0 & \text{otherwise}.
\end{cases}
\end{equation*}
\end{lem}
\begin{proof}[Sketch proof]
We first observe that $k_{km-1}\Lambda^k M \subseteq k_{km-1}M^{\otimes k} = 0$, so that $\bar k_{km}\Lambda^k M = \Lambda^k M$.

Now we recall that $U(M) = S^*(M)/(Sq^{|x|}x-x^2)$.  This is filtered with $U^k(M)$ equal to the image of $M^{\otimes k}$ in $U(M)$, and there are exact sequences
$$ 0 \ra U^{k-1}(M) \ra U^{k}(M) \ra \Lambda^{k} M \ra 0.$$

Applying $k_n$ to this yields exact sequences
$$ 0 \ra k_nU^{k-1}(M) \ra k_nU^{k}(M) \ra k_n\Lambda^{k} M,$$
and one deduces that
$$ U^{k-1}(M) = k_{km-1}U^k(M) \text{ so that } \bar k_{km}U^k(M) = \Lambda^k M$$
and
$$ k_{km+r}U^k(M) = k_{km+r}U(M)$$
for $0\leq r < m$.
\end{proof}

\begin{lem}  If $\bT^{m+1}M = 0$, then $\bar T^{km}\Lambda^k M = \Ind_{\Sigma_k \wr \Sigma_m}^{\Sigma_{km}}((\bT^mM)^{\otimes k}))$.
\end{lem}
\begin{proof}[Sketch proof]  This follows by iterated use of the following facts: $T\Lambda^kM = \Lambda^kTM$, $TM = \bT M \oplus M$, and $\Lambda^k(M \otimes N) = \bigoplus_{i+k=k} \Lambda^i M \otimes \Lambda^j N$.
\end{proof}

The proposition is illustrated with the next example.

\begin{ex} $\sigma_*(H^*(K(V,m))) = Sh^m_*(V^{\vee})$, where $V^{\vee}$ is the dual of $V$, and is given trivial $\Sigma_m$--module structure.
\end{ex}

We end with a group cohomology example.

\begin{ex} Let $Q_8$ be the quaternionic group of order 8. Then
$$\sigma_*(H^*(BQ_8)) = H^*(S^3/Q_8) \otimes Sh^1_*(\Z/2),$$  with $\sigma_0(H^*(BQ_8)) = H^*(S^3/Q_8)  = \Z/2[x,y]/(x^2+xy+y^2,x^2y+xy^2)$.

This can be seen quite easily.  $Q_8 \subset S^3$ induces an embedding
$$\Phi^2 H^*(B\Z/2) = H^*(BS^3) \subset H^*(BQ_8).$$  The image in cohomology of  $Q_8 \ra Q_8/[Q_8,Q_8] = (\Z/2)^2$ is $\sigma_0H^*(BQ_8)$, and the composite
$\sigma_0H^*(BQ_8) \hra H^*(BQ_8) \ra H^*(S^3/Q_8)$ is an isomorphism.  These maps combine to define an isomorphism in $\K$:
$$ H^*(S^3/Q_8) \otimes \Phi^2 H^*(B\Z/2)  \simeq H^*(BQ_8),$$
and the calculation of $\sigma_*H^*(BQ_8)$ follows.
\end{ex}

\end{document}